\documentclass[a4paper,12pt,reqno]{amsart}
\usepackage{amsfonts,amssymb,amsmath,hyperref,amsthm,enumerate,
color,stmaryrd,pgf,tikz,comment,multirow}
\usepackage[left=2.6cm,right=2.6cm,top=3cm,bottom=3cm,bindingoffset=0cm]
{geometry}

\usepackage{array}   \usepackage{multirow}  \usepackage{booktabs,caption,fixltx2e}
\usepackage[flushleft]{threeparttable}

\newtheorem{theorem}{Theorem}[section]
\newtheorem{lemma}[theorem]{Lemma}
\newtheorem{proposition}[theorem]{Proposition}

\theoremstyle{definition}

\newtheorem{example}[theorem]{Example}

\newtheorem{problem}{Problem}

\newtheorem{definition}[theorem]{Definition}

\def\Z{\mathbb{Z}}
\DeclareMathOperator{\Aut}{Aut}
\DeclareMathOperator{\Ker}{Ker}

\DeclareMathOperator{\Skew}{Skew}
\DeclareMathOperator{\Ext}{Ext}

\DeclareMathOperator{\Core}{Core}

\DeclareMathOperator{\av}{av}
\DeclareMathOperator{\Av}{Av}

\title[Cyclic complementary  extensions and skew-morphisms]
{Cyclic complementary  extensions and skew-morphisms}
\author[K.~Hu\and R. Jajcay ]
{Kan Hu\and Robert Jajcay }
\address{K. Hu,
\newline\indent
Department of Mathematics, Zhejiang Ocean University, Zhoushan, Zhejiang 316022, P.R. China
\newline\indent UP FAMNIT, University of Primorska, Glagolja\v ska 8,
6000 Koper, Slovenia
\newline\indent UP IAM, University of Primorska, Muzejski trg 2,
6000 Koper, Slovenia}
\email{hukan@zjou.edu.cn}

\address{R. Jajcay
\newline\indent
Department of Algebra and Geometry, Faculty of Mathematics, Physics and Computer Science, Comenius University, Bratislava, Slovakia}
\email{robert.jajcay@fmph.uniba.sk}

\thanks{}
\keywords{group extension, complementary factorization, skew-morphism}
\subjclass[2010]{05C10, 05C25, 57M15}
\begin{document}
\maketitle

\begin{abstract}
A cyclic complementary extension of a finite group $A$  is a finite group $G$ which contains $A$ and a cyclic subgroup $C$ such that $A\cap C=\{1_G\}$ and $G=AC$.
For any fixed generator $c$ of the cyclic factor  $C=\langle c\rangle$ of order $n$ in a 
cyclic complementary extension $G=AC$,
the equations $cx=\varphi(x)c^{\Pi(x)}$, $x\in A$, determine a permutation $\varphi:A\to A$ 
 and a function $\Pi:A\to\mathbb{Z}_n$ on $A$ characterized by the properties:
 (a) $\varphi(1_A)=1_A$ and $\Pi(1_A)\equiv1\pmod{n}$; (b) $\varphi(xy)=\varphi(x)\varphi^{\Pi(x)}(y)$
 and $\Pi(xy)\equiv\sum_{i=1}^{\Pi(x)}\Pi(\varphi^{i-1}(y))\pmod{n}$, for all $x,y\in A$. 
The permutation $\varphi$ is called a skew-morphism of $A$ and has already been extensively studied. One of the main contributions of the present paper is the recognition
of the importance of the function $\Pi$, which we call the extended power function associated with $\varphi$. We show that {\em every} cyclic complementary extension of $A$ is determined and can be constructed from a
skew-morphism $\varphi$ of $A$ and an extended power function $\Pi$ associated with $\varphi$.
As an application, we present a classification of cyclic complementary extensions of cyclic groups
obtained using skew-morphisms which are group automorphisms.
\end{abstract}
\section{Introduction}\label{sec:intro}

A group $G$ is defined to be an \textit{exact product} of two subgroups $A$ and $B$ if $A\cap B=\langle 1_G \rangle $ and $G=AB$.  A classical theorem by It\^o establishes that an exact product of two abelian groups is metabelian~\cite{Ito1955}. Additionally, it follows from a theorem by Wielandt and Kegal~\cite{Wielandt1958,Kegel1961} that an exact product of two nilpotent groups is solvable. Exact products have been subject to intensive research, particularly in connection with primitive permutation groups containing a regular subgroup, and finite simple groups~\cite{LWX2023,LPS1996,LPS2010}.

A  {\em  cyclic complementary extension} of a finite group $A$ is a finite group $G =AC$ which
is an exact product of the subgroup $A$ and a cyclic subgroup $C$. Thus, the notion of a cyclic complementary extension  can be regarded as a specific instance of an exact product of two subgroups, one of which is cyclic.

Cyclic complementary extensions of cyclic groups, which constitute a special case of the more
general concept of {\em bicyclic groups}, i.e., products of two cyclic subgroups, 
have been intensely investigated by group theorists. A theorem of Douglas states that every 
bicyclic group is supersolvable~\cite{Douglas1961}. Huppert's theorem asserts that, for any odd
prime $p$, every bicyclic $p$-group is metacyclic~\cite{Huppert1953}. 
More recently, Janko contributed to this body of knowledge by providing a characterization of non-metacyclic 
bicyclic $2$-groups. His work revealed that a non-metacyclic $2$-group is bicyclic if and only if it has 
a rank $2$ and contains precisely one non-metacyclic maximal subgroup~\cite{Janko2008}.

In this paper, we study cyclic complementary  extensions using the language of skew-morphisms which are
special permutations of elements of a group and have been shown to play
a crucial role in the context of regular Cayley maps; see for example~\cite{BCV2022, BJ2017,CDL2022,  CJT2016,CT2014,DYL2022,HKK2022,JS2002,KK2021,KN2017}.

Given a (finite) group $A$, a {\em skew-morphism\/} of $A$ is a permutation of elements
of $A$, $\varphi: A \to A$, that fixes the identity $ 1_A $ of $A$ and admits the existence of an 
{\em associated power function} $\pi\!:A \to \Z$ such that 
$$ \varphi(xy) = \varphi(x)\varphi^{\pi(x)}(y) \ \hbox{ for all }x,y \in A.$$  
The {\em order} $m$ of $ \varphi $ is the usual order of $ \varphi $ as a permutation under 
composition, and the power function $\pi$ can be regarded as a function from $A$ to $\Z_m$. 
A skew-morphism $\varphi$ is an automorphism of $A$ if and only if, the value of the power function 
$\pi$ is congruent modulo $m$ to $1$ for all $x \in A$. In case of skew-morphisms which are not 
automorphisms, the {\em kernel\/} of $\varphi$ is defined as the subset $\{ x \in A \mid \pi(x) \equiv1\pmod{m} \}$, 
and is denoted by $\Ker\pi$.  The kernel $ \Ker\pi$ is known to be a subgroup of $A$, and 
the values of $\pi$ on two elements of $ x,y \in A$ coincide if and only if $x$ and $y$ 
lie in the same right coset of $ \Ker\pi$.

The connection between skew-morphisms and cyclic complementary extensions 
 stems from the following observations. If $G=A\langle c\rangle$ is a cyclic complementary extension of 
 a subgroup $A$  by a cyclic subgroup group $\langle c\rangle$ of order $n$, then $ A \langle c \rangle=G  = \langle c\rangle A $,   so for each $x \in A $ there exist  
$ x' \in A $ and $ i \in {\Bbb Z}_n $ such that $ cx = x'c^i $, 
both of which are uniquely determined by $x$. 
This uniqueness allows us to define functions $\varphi \! : A \to A$ and $\Pi \!: A \to {\Bbb Z}_n$ 
by setting 
\begin{equation}\label{skew-def}
\varphi(x) = x' \ \hbox{ and } \ \Pi(x) = i \ \hbox{ whenever} \ cx = x'c^i, \ \hbox{ where } 
  x' \in A \hbox{ and } i \in {\Bbb Z}_n.
\end{equation}
It is well-known~\cite{CJT2016} that the mapping $\varphi $ defined this way is a skew-morphism   
of $A$ of order $m:=|\varphi|$ equal to the index $|\langle c\rangle:\langle c\rangle_G|$, where $\langle c\rangle_G$ 
denotes the core of $\langle c\rangle$ in $G$. The associated power function $\pi:A\to\mathbb{Z}_{m}$ of $\varphi$ is determined by 
\begin{equation}\label{PF}
\pi(x)\equiv\Pi(x)\pmod{m},\quad\text{for all $x\in A$}.
\end{equation}
In the forthcoming, we show that the function $\Pi:A\to\mathbb{Z}_n$ defined by the equations
\eqref{skew-def} has the following defining properties (see Section 3): 
\begin{enumerate}[\rm(a)]
\item $\Pi(1)\equiv1\pmod{n}$,
 \item $\Pi(xy)\equiv\sum_{i=1}^{\Pi(x)}\Pi(\varphi^{i-1}(y))\pmod{n}$ for all $x,y\in A$.
 \end{enumerate}
Using the function $\Pi:A\to\mathbb{Z}_n$, the left multiplication of the elements
of $A$ by $c$ can now be reformulated to  
\begin{equation}\label{Comm}
cx = \varphi(x)c^{\Pi(x)},\quad\text{ for all $x \in A$}.
\end{equation}
Thus, every cyclic complementary  extension of $A$ yields a skew-morphism of $A$. This 
skew-morphism is not necessarily unique, as 
different choices of a generator for $ \langle c \rangle $ usually yield different skew-morphisms.

To take the opposite point of view on this correspondence, assume that $ \varphi : A \to
A $ is a skew-morphism of $A$ of order $m$ with the associated power function $ \pi :
A \to {\Bbb Z}_m $. The {\em skew product} $A \langle \varphi \rangle $ of $A$ with $ \varphi $ 
of order $m$ is defined in~\cite{CJT2016} 
to be the group whose elements are pairs of elements in $ A \times \langle \varphi \rangle $ under the multiplication
\begin{eqnarray}\label{skew-product} 
x \varphi^i \cdot y \varphi^j = x \varphi^i(y) \varphi^{\sum_{t=1}^i\pi(\varphi^{t-1}(y))+j} , 
\quad\mbox{ for all } x,y \in A\quad\text{and}\quad i,j \in {\Bbb Z}_m .
\end{eqnarray}
Clearly, each skew product of a group $A$ with one of its skew-morphisms is a 
cyclic complementary  extension of $A$ of order $ |A| \cdot m $; where $m$ is the order
of $ \varphi $. This might be viewed as
the reversed correspondence between cyclic complementary  
extensions of a group and its skew-morphisms. 

There is, however, a gap in this picture, namely, not every cyclic complementary  
extension of $A$ is a skew product of $A$ with one of its skew-morphisms. 
Indeed, since the cyclic factor $\langle\varphi\rangle$
is core-free in the skew product $A\langle\varphi\rangle$, a cyclic
complementary extension  $G=A\langle c\rangle$  of $A$ is a skew product of  skew-morphism $ \varphi $ 
of $A$ if and only if the cyclic complement $\langle c\rangle$ of $A$ 
is core-free in $G$, or equivalently,  the order of $ \varphi $ 
 matches the order of  $c$. Another notable fact, as shown in \cite[Theorem 4.2]{CJT2016}, 
is that $|\varphi|<|A|$ whenever $|A| \neq 1$. Hence, no cyclic complementary  
extension $G=A\langle c\rangle$ in which $|c| \geq |A| > 1$
is a skew-product of $A$.

One of the aims of this paper is to fill this gap by  determining all cyclic complementary 
extensions $G = A \langle c \rangle $ of $A$ which correspond to a prescribed 
 skew-morphism $ \varphi $. More precisely, we consider the following problem:
 \begin{problem}\label{main-pro}
 For a given a skew-morphism $\varphi$ of a finite group $A$, characterize all cyclic complementary
 extensions $G=A\langle c\rangle$ of $A$ such that the skew-morphism determined by \eqref{skew-def}
 is the prescribed skew-morphism $\varphi$.
 \end{problem}
To solve Problem~\ref{main-pro}, we  first generalize the concept of a power function $\pi:A\to\mathbb{Z}_m$
of a skew-morphism $\varphi$ to the  concept of an extended power function of $\varphi$,
which is a function $\Pi:A\to\mathbb{Z}_n$ defined for a positive multiple $n$ of  $m=|\varphi|$
having the following properties: 
\begin{enumerate}[\rm(a)]
\item $\Pi(x)\equiv\pi(x)\pmod{m}$ for all $x\in A$,
\item $\Pi(1)\equiv1\pmod{n}$,
 \item $\Pi(xy)\equiv\sum_{i=1}^{\Pi(x)}\Pi(\varphi^{i-1}(y))\pmod{n}$ for all $x,y\in A$.
 \end{enumerate}
 The first main result of this paper is the following correspondence theorem; see Section~\ref{sec:esp} 
 for its proof.
 \begin{theorem}\label{main-thm1}
 If $G=A\langle c\rangle$ is a cyclic complementary extension of a group $A$ by a cyclic group $\langle c\rangle$ 
 of order $n$, then the equation $cx=\varphi(x)c^{\Pi(x)}$,
 $x\in A$, determines a pair $(\varphi,\Pi)$ of a skew-morphism $\varphi$ of $A$ and an extended power function 
 $\Pi:A\to\mathbb{Z}_n$ of $\varphi$.
 
Conversely, every pair $(\varphi,\Pi)$ of a skew-morphism $\varphi$ of a group $A$ and 
an extended power function $\Pi:A\to\mathbb{Z}_n$ of $\varphi$ determines
 a cyclic complementary extension $A\langle c\rangle$ of $A$ by a cyclic group $\langle c\rangle$ of order $n$
such that the skew-morphism determined by the equation $cx=\varphi(x)c^{\Pi(x)}$,
 $x\in A$,  is $\varphi$.

Moreover, suppose that $(\varphi_i,\Pi_i)$ are two such pairs,
and $A\langle c_i\rangle$  are the corresponding cyclic complementary extensions of $A$, $i=1,2$.
Then  there exists an isomorphism $\Theta:A\langle c_1\rangle\to A\langle c_2\rangle$ such that $\Theta(A)=A$ and 
 $\Theta(c_1)=c_2$ if and only if there is an automorphism $\theta$ of $A$ such that $\varphi_2=\theta\varphi_1\theta^{-1}$ and  $\Pi_2=\Pi_1\theta^{-1}$.
\end{theorem}

In view of Theorem~\ref{main-thm1}, the problem of determining all cyclic complementary extensions  
of a group $A$ amounts to determining all pairs $(\varphi,\Pi)$ consisting of a skew-morphism 
$\varphi$ of $A$ and an extended power function $\Pi$ of $\varphi$. This adds further importance to
the ongoing research into the classification of regular Cayley maps, skew-morphisms, and skew-product groups
 for various classes of finite groups. To mention just a few of these results, we note that 
these classification problems have been investigated for example for cyclic groups~\cite{CT2014, JS2002,CJT2007,DH2019, KN2011, KN2017}, abelian groups \cite{CJT2007,CJT2016}, dihedral groups~\cite{ZD2016,HKK2023,KK2021}, 
non-abelian simple groups~\cite{BCV2022}, and characteristically simple groups~\cite{CDL2022,DLY2022}. 

Further results concerning cyclic groups tend to classify 
skew-morphisms with additional properties, such as smooth skew-morphisms of cyclic groups \cite{BJ2017}, or
 skew-morphisms of cyclic groups which are square roots of automorphisms~\cite{HKZ2021}.
In addition, certain classes of finite groups have been shown to possess only skew-morphisms of
limited types. Specifically, it is known that every skew-morphism of a cyclic group of order $n$ 
is an automorphism if and only if $\gcd(n,\phi(n))=1$~\cite{JS2002, KN2011},
and every skew-morphism of a non-cyclic abelian group is an automorphism if and only if it is an
elementary abelian $2$-group~\cite{CJT2007}. An interesting generalization has been recently obtained in~\cite{BACH2023, HKK2023}, where it has been shown
that every skew-morphism of a cyclic groups of order $n$ is smooth
if and only if $n=2^en_1$, where $0\leq e\leq 4$, and $n_1$ is odd and square-free~\cite{BACH2023, HKK2023}.

Any classification of cyclic complementary extensions for a fixed group $A$, even in the case
when we know all of its skew-morphisms, would require the classification of all extended power 
functions for each of its skew-morphisms. Extended power functions, being a new topic introduced
in the present paper, have so far not been studied at all, and it is our belief that finding all extended
power functions for a given skew-morphism will require building an entire theory of extended power
functions. As a first step toward this goal, in the following theorem we classify all cyclic complementary extensions of the cyclic group $A=\langle a|a^k=1\rangle$ corresponding to automorphisms of $A$; 
the proof of this theorem will be given in Section~5.

\begin{theorem}\label{main-thm2}
Let $\varphi:a\mapsto a^r$ be an automorphism of the cyclic group $A=\langle a|a^k=1\rangle$
of order $k$, where $1\leq r<k$ and $\gcd(r,k)=1$, and let $n$ be a positive multiple of the multiplicative order $m$
 of $r$ in $\mathbb{Z}_k$. Then, every cyclic complementary extension $G=A\langle c\rangle$ of $A$ 
 by a cyclic group $C=\langle c| c^n=1\rangle$ of order $n$ 
determined by the automorphism $\varphi$ has a presentation
\[
G=\langle a,c|a^k=c^n=1, c^ma=ac^{mt}, ca=a^rc^{1+ms}\rangle,
\]
where the parameters $r\in\mathbb{Z}_k^*$ and $s,t\in\mathbb{Z}_{n/m}$ satisfy the conditions:
\begin{enumerate}[\rm(a)]
\item $t^{r-1}\equiv 1\pmod{n/m}$,
\item $s\sum_{i=1}^kt^{i-1}\equiv0\pmod{n/m}$,
\item $s\sum_{i=1}^m(\sum_{j=1}^rt^{j-1})^{i-1}\equiv t-1\pmod{n/m}$.
\end{enumerate}
\end{theorem}

Our paper is organized as follows. In Section~\ref{sec:dpf} we list several known results as well as 
prove some new ones concerning power functions of skew-morphisms. In the next Section~\ref{sec:epf}, we define
the key concept of an extended power function and derive some of its fundamental
properties. 
Theorem~\ref{main-thm1} is proved in Section~\ref{sec:esp} 
where a correspondence is established between cyclic complementary extensions of a group $A$ and pairs 
consisting of a skew-morphism $\varphi$ of $A$ and an extended power function $\Pi$ of $\varphi$. As 
an application
of the theory developed in Sections~\ref{sec:epf}~and~\ref{sec:esp}, in Section~\ref{sec:app} we determine cyclic complementary extensions of cyclic groups
corresponding to automorphisms. The paper is closed with a section
devoted to open problems and directions for future research.
\section{Derived power functions}\label{sec:dpf}
In this section we state some preliminary terminology and results 
on skew-morphism for future reference.

Let  $\varphi$ be a skew-morphism of a finite group $A$ of order $m:=|\varphi|$, 
with the associated power function $\pi : A \to {\Bbb Z}_m $. Given any $f:A\to\mathbb{N}$, 
we may define a {\em derived function} $\sigma_{f}:A\times\mathbb{N}\to\mathbb{N}$  by
 \begin{align*}
 &\sigma_f(x,0)=0,\quad\sigma_f(x,k)=\sum\limits_{i=1}^{k-1}f(\varphi^{i-1}(x)),~k>0.
 \end{align*}
Specifically, the {\em derived function $\sigma_{\pi}$ of the power function} $\pi:A\to\mathbb{Z}_m$ is defined by
 \begin{align}\label{Expf1}
\sigma_{\pi}(x,0)=0,\quad \sigma_{\pi}(x,k)=\sum\limits_{i=1}^{k-1}\pi(\varphi^{i-1}(x)), \quad k>0.
 \end{align}
It will be referred to as the \textit{derived power function} of $\varphi$.

It is well known that the subset
\begin{align*}
\Ker\pi:=\{x\in A\mid \pi(x)\equiv1\pmod{m}\}
\end{align*}
is a subgroup of $A$, called the \textit{kernel} of $\pi$~\cite{JS2002}.  The following properties of the skew-morphisms, their power
functions, and their derived power functions have already been used in a large number of papers in this area.

\begin{proposition}[\cite{HKZ2021, JS2002}]\label{Basic}
Let $\varphi$ be a skew-morphism of a group $A$ of order $m:=|\varphi|$, 
$\pi$ be the power function of $\varphi$, and $\sigma_{\pi}$ be its derived power function.
Then the following statements hold true:
\begin{enumerate}[\rm(a)]
\item For any $x,y\in A$, $\pi(x)\equiv\pi(y)\pmod{m}$ if and only if $xy^{-1}\in\Ker\pi$.
\item $\pi(xy)\equiv\sigma_{\pi}(y,\pi(x))\pmod{m}$  for any $x,y\in A$.
\item $\sigma_{\pi}(x,m)\equiv0\pmod{m}$ for all $x\in A$.
\end{enumerate}
Moreover, suppose that $k$ is an arbitrary nonnegative integer. Then
\begin{enumerate}[\rm(d)]
\item[\rm(d)]  $\varphi^{k}(xy)=\varphi^{k}(x)\varphi^{\sigma_{\pi}(x,k)}(y)$  for any $x,y\in A$.
\item[\rm(e)]  $\mu:=\varphi^k$ is a skew-morphism of $A$ if and only if there exists a
function $\pi_{\mu}:A\to\mathbb{Z}_{m/\gcd(m,k)}$ such that $k\pi_{\mu}(x)\equiv\sigma_{\pi}(x,k)\pmod{m}$
for all $x\in A$.
\item[\rm(f)] $\varphi^k$ is an automorphism of $A$ if and only if $\sigma_{\pi}(x,k)\equiv k\pmod{m}$ for all $x\in A$.
\end{enumerate}
\end{proposition}

Note that $\Ker\pi$ is not necessarily a $\varphi$-invariant subgroup of $A$.
 However, the subset
\[
\Core\pi:=\bigcap_{i=1}^m\varphi^i(\Ker\pi)
\]
is the largest $\varphi$-invariant subgroup of $A$ contained in $\Ker\pi$~\cite{Zhang2015b}, and it is 
termed the \textit{core} of $\pi$. It can be shown that 
$\Core\pi=\bigcap_{g\in G}A^g$, where $G=A\langle\varphi\rangle$ is the skew product of $A$ with $\varphi$~\cite[Proposition~6]{HNWY2019}. Thus,  in particular, $\Core\pi$ is normal in $A$.

In general, if $N$ is a $\varphi$-invariant normal subgroup of $A$, then we may define a permutation $\bar\varphi$
 of the quotient group $\bar A:=A/N$ by
 \[
\bar \varphi(\bar x)=\overline{\varphi(x)},\quad x\in A.
 \]
 It is known that $\bar\varphi$ is a skew-morphism of $\bar A$~\cite{ZD2016} with the associated 
 power function $\bar\pi$ determined by $\bar\pi(\bar x)\equiv\pi(x)\pmod{|\bar\varphi|}$, $x\in A$. 
It will be referred to as the \textit{quotient skew-morphism} of $\varphi$ induced by $N$.

The \textit{period} of the power function $\pi:A\to\mathbb{Z}_m$ of a skew-morphism $\varphi$
 is defined to be the smallest positive integer $p$ such that 
\[
\pi(\varphi^{p}(x))\equiv\pi(x)\pmod{m},\quad\text{ for all $x\in A$}.
\]
A skew-morphisms of period $1$ is termed \textit{smooth} (or~\textit{coset-preserving}). 
Further properties of smooth skew-morphisms can be found in~~\cite{BJ2017, WHYZ2019}.

The following result reveals the relationship between the period of $\pi$ and the quotient skew-morphism of $\varphi$ induced by $\Core\pi$.
\begin{proposition}[\cite{WHYZ2019}]\label{smooth-basic}
Let $\varphi$ be a skew-morphism of a group $A$ with the associated power function 
$\pi:A\to\mathbb{Z}_m$, where $m=|\varphi|$,  and let $\bar\varphi$ 
be the quotient skew-morphism of $\varphi$ induced by $\Core\pi$, and let $p$ be the period of $\varphi$.
Then  the following statements hold:
\begin{enumerate}[\rm(a)]
\item $p$ divides $m$ and is equal to the order  of $\bar\varphi$.
\item $\varphi^{p}$ is a smooth skew-morphism of $A$ of order $m/p$.
\item For any positive integer $k$, $\sigma_{\pi}(x,k)\equiv 0\pmod{p}$ for any $x\in A$ if and only if $p|k$.
\end{enumerate}
\end{proposition}
The power function associated with the skew-morphism $\varphi^p$ in 
Proposition~\ref{smooth-basic}(b) is particularly important in the study of $\varphi$. It was first investigated in~\cite{HR2022}
under the name of the \textit{average function} of $\pi$, and we will denote it by $\av:A\to\mathbb{Z}_{m/p}$. 
It is defined via the formula
\[
\av(x)=\frac{1}{p}\sigma_{\pi}(x,m),\quad x\in A.
\]

Moreover, since the power function associated with the quotient skew-morphism $\bar\varphi$ of $\varphi$ 
induced by $\Core\pi$ is determined by $\bar\pi(\bar x)\equiv\pi(x)\pmod{p}$, $x\in A$, we may define
a function $\lambda:A\to\mathbb{Z}_{m/p}$ by 
\[
\lambda(x)=\frac{1}{p}\big(\pi(x)-\bar\pi(\bar x)\big),\quad x\in A.
\]
It will be called the \textit{mate function} of $\pi$ (\textit{with respect to} $\bar\pi$). The above
defined functions have the following important properties.

\begin{proposition}\label{av-lambda}
Let $\varphi$ be a skew-morphism of a group $A$ with the associated power function
 $\pi:A\to\mathbb{Z}_m$, where $m$ is the order of $\varphi$, let $\bar\varphi$ be the 
 quotient skew-morphism of $\bar A:=A/\Core\pi$ induced by $\Core\pi$, and let $\av:A\to\mathbb{Z}_{m/p}$
 and $\lambda:A\to\mathbb{Z}_{m/p}$ be the average function and the mate function of $\pi$, respectively. 
 Then the following statements hold true:
 \begin{enumerate}[\rm(a)]
 \item $\av(\varphi(x))\equiv\av(x)\pmod{m/p}$ for all $x\in A$.
 \item $\av$ is a homomorphism from $A$ into the multiplicative group $\mathbb{Z}_{m/p}^*$.
 \item If $\bar\varphi$ is an automorphism of $\bar A$, then $\lambda(xy)=\av(y)\lambda(x)+\lambda(y)$, for all $x,y\in A$.
 In particular, if $\av(x)\equiv1\pmod{n/m}$ for all $x\in A$, then $\lambda$ is a homomorphism from $A$ into
 the additive group $\mathbb{Z}_{m/p}$.
 \end{enumerate}
\end{proposition}
\begin{proof}

The proof of (a) and (b) can be found in \cite{HR2022}. We give a proof of (c). Since $\bar\varphi$
is an automorphism of $\bar A$, we have $\bar\pi(\bar x)\equiv1\pmod{p}$ for all $x\in A$, so
\[
\lambda(x)=\frac{1}{p}(\pi(x)-1)\pmod{m/p}.
\]
 Thus, for all $x,y\in A$, we have
\begin{align*}
\lambda(xy)&\equiv\frac{1}{p}(\pi(xy)-1)\equiv\frac{1}{p}\big(\sigma_{\pi}(y,\pi(x))-1\big)\equiv\frac{1}{p}\big(\pi(y)-1+\frac{\pi(x)-1}{p}\sigma_{\pi}(y,p)\big)\\
&\equiv\lambda(y)+\lambda(x)\av(y)\pmod{m/p}.
\end{align*}
In particular,  if $\av(x)\equiv1\pmod{m/p}$, for all $x\in A$, then the equation reduces
to $\lambda(xy)\equiv\lambda(x)+\lambda(y)\pmod{m/p}$, so $\lambda$ is 
a homomorphism of $A$ into the additive group 
$\mathbb{Z}_{n/m}$, as required.
\end{proof}

Now we turn to the derived power function $\sigma_{\pi}$ of $\pi$. It  has the following properties.
\begin{proposition}\label{SIGMA}
Let  $\varphi$ be a skew-morphism of a group $A$ of order $m$ with the
associated  derived power function $\sigma_{\pi}: A \times {\Bbb N} \to {\Bbb N}  $. 
Then $\sigma_{\pi}$ has the following properties:
\begin{enumerate}[\rm(a)]
\item $\sigma_{\pi}(xy,k)\equiv\sigma_{\pi}(y,\sigma_{\pi}(x,k))\pmod{m}$  for all $x,y\in A$.
\item $\sigma_{\pi}(x,k)\equiv0\pmod{m}$ for all $x\in A$ if and only if $m|k$.
\item  $\sigma_{\pi}(x,mq+r) \equiv \sigma_{\pi}(x,r) \pmod{m} $ for all $x\in A$.
\item $\sigma_{\pi}(x,k_1)\equiv\sigma_{\pi}(x,k_2)\pmod{m}$ for all $x\in A$  if and only if $k_1\equiv k_2\pmod{m}$.
\end{enumerate}
\end{proposition}
\begin{proof}(a) By Proposition~\ref{Basic}(b), we have
\begin{align*}
\sigma_{\pi}(xy,k)&=\sum_{i=1}^k\pi(\varphi^{i-1}(xy))=\sum_{i=1}^k\pi\big(\varphi^{i-1}(x)\varphi^{\sigma_{\pi}(x,i-1)}(y)\big)
\equiv\sum_{i=1}^k\sum_{j=1}^{\pi(\varphi^{i-1}(x))}\pi\big(\varphi^{\sigma_{\pi}(x,i-1)+j-1}(y)\big)\\
&=\sum_{j=1}^{\pi(x)}\pi(\varphi^{j-1}(y))+\sum_{j=1}^{\pi(\varphi(x))}\pi\big(\varphi^{\sigma_{\pi}(x,1)+j-1}(y)\big)+\cdots+\sum_{j=1}^{\pi(\varphi^{k-1}(x))}\pi(\varphi^{\sigma_{\pi}(x,k-1)+j-1}(y))\\
&=\sum_{j=1}^{\sigma_{\pi}(x,k)}\pi(\varphi^{j-1}(y))=\sigma_{\pi}(y,\sigma_{\pi}(x,k))\pmod{m}.
\end{align*}

(b) If $\sigma_{\pi}(x,k)\equiv0\pmod{m}$ for all $x\in A$, then
\[
1=\varphi^k(xx^{-1})\equiv\varphi^k(x)\varphi^{\sigma_{\pi}(x,k)}(x^{-1})=\varphi^k(x)x^{-1},
\]
and so $\varphi^k(x)=x$, implying that $|\mathcal{O}_x|$ divides $k$. 
Since $m$ is equal to the least common multiple of $|\mathcal{O}_x|$, $x\in A$, we obtain  $m|k$.

Conversely, if $m|k$, then we may write $k=mk_1$. For any $x\in A$, we have $\sigma_{\pi}(x,m)\equiv0\pmod{m}$
by Proposition~\ref{Basic}(c), so
\[
\sigma_{\pi}(x,k)\equiv\sigma_{\pi}(x,mk_1)\equiv k_1\sigma_{\pi}(x,m)\equiv0\pmod{k}.
\]

(c) By Proposition~\ref{Basic}(c), we have
\begin{align*}
\sigma_{\pi}(x,mq+r)&=\sum_{i=1}^{mq+r}\pi(\varphi^{i-1}(x))=\sum_{i=1}^r\pi(\varphi^{i-1}(x))+\sum_{i=r+1}^{mq+r}\pi(\varphi^{i-1}(x))\\
&=\sigma_{\pi}(x,r)+\sum_{i=1}^{mq}\pi(\varphi^{i-1}(x))=\sigma_{\pi}(x,r)+q\sigma_{\pi}(x,m) \equiv \sigma_{\pi}(x,r) \pmod{n}.
\end{align*}

(d) If $k_1\equiv k_2\pmod{m}$, then from (b) we deduce $\sigma_{\pi}(x,k_1)\equiv\sigma_{\pi}(x,k_2)\pmod{m}$
for all $x\in A$. Conversely, if $\sigma_{\pi}(x,k_1)\equiv\sigma_{\pi}(x,k_2)\pmod{m}$
for all $x\in A$, then we may assume $k_1\leq k_2$, and denote $y=\varphi^{k_1}(x)$. Then $y$ runs through $A$, and
\[
\sigma_{\pi}(y,k_2-k_1)=\sigma_{\pi}(x,k_2)-\sigma_{\pi}(x,k_1)\equiv0\pmod{m}.
\]
Thus, by (b) we have $k_1\equiv k_2\pmod{m}$.
\end{proof}

\section{Extended power functions}\label{sec:epf}
In this section, we formally define the concept of an extended power function 
of a skew-morphism $\varphi$, and develop a theory for further applications.

Let  $G=A\langle c\rangle$ be a cyclic complementary extension of a finite group $A$
by a cyclic subgroup $\langle c\rangle$ of order $n$, and let $\varphi$  be 
the skew-morphism of $A$ and $\Pi:A\to\mathbb{Z}_{n}$ the associated
power function of $\varphi$ determined as in~\eqref{skew-def}. Then, 
for any positive integer $k$, by repeated use of
 the commuting rule~\eqref{Comm}, we have
 \[
 c^kx=\varphi^k(x)c^{\sigma_{\Pi}(x,k)},\quad\text{for any $x\in A$},
 \]
 where $\sigma_{\Pi}$ is the derived function of $\Pi$.
 It is evident that $\Pi(1_A)\equiv1\pmod{n}$, and $\Pi(x)\equiv\pi(x)\pmod{m}$ for all $x\in A$. Moreover, for any $x,y\in A$, we have
\[
\varphi(xy)c^{\Pi(xy)}=c(xy)=(cx)y=\varphi(x)c^{\Pi(x)}y=\varphi(x)\varphi^{\Pi(x)}(y)c^{\sigma_{\Pi}(y,\Pi(x))},
\]
so  $\Pi(xy)\equiv\sigma_{\Pi}(y,\Pi(x))\pmod{n}$.

Based on these observations, we define the extended power function of a skew-morphism 
$\varphi$ as a function that satisfies the above derived properties.
\begin{definition}\label{n-power}
Let $\varphi$ be a skew-morphism of a finite group $A$ with power function $\pi:A\to\mathbb{Z}_m$, where
$m=|\varphi|$. For a positive multiple $n$ of $m$,  a function $\Pi:A\to\mathbb{Z}_n$ will be called
an \textit{extended power function} of $\varphi$ if it satisfies the following conditions:
\begin{enumerate}[\rm(a)]
\item $\Pi(x)\equiv\pi(x)\pmod{m}$ for all $x\in A$,
\item $\Pi(1_A)\equiv1\pmod{n}$,
 \item$\Pi(xy)\equiv\sum_{i=1}^{\Pi(x)}\Pi(\varphi^{i-1}(y))\pmod{n}$ for all $x,y\in A$.
 \end{enumerate}
\end{definition}

By examining the definition of an extended power function, it becomes evident that this concept is a natural extension of the concept of a power function associated with $\varphi$. Therefore, it should come as no surprise that beside the properties listed in Definition~\ref{n-power}, the two exhibit numerous additional analogous properties. For instance, we can define the kernel and the core of $\Pi$, respectively, to be the subsets
 \[
\Ker\Pi=\{x\in A\mid \Pi(x)\equiv1\pmod{n}\}\quad\text{and}\quad \Core\Pi=\bigcap_{i=1}^m\varphi^{i-1}(\Ker\Pi).
\]
It is clear that both $\Ker\Pi$ and $\Core\Pi$ are subgroups of $A$. In particular,  $\Ker\Pi\leq\Ker\pi$ and $\Core\Pi\leq\Core\pi$.
Moreover, the following lemma shows that the functions $\Pi$ and $\sigma_{\Pi}$ have similar
properties as $\pi$ and $\sigma_{\pi}$.

\begin{proposition}\label{SIGMA2}
Let $\varphi$ be a skew-morphism of a group $A$ with the associated 
power function $\pi:A\to\mathbb{Z}_m$, where $m=|\varphi|$ and $n$ is a positive multiple of $m$, 
let $\Pi:A\to\mathbb{Z}_n$ be an extended power function of $\varphi$. Then the following 
statements hold true:
\begin{enumerate}[\rm(a)]
\item $\sigma_{\Pi}(xy,k)\equiv\sigma_{\Pi}(y,\sigma_{\Pi}(x,k))\pmod{n}$  for all $x,y\in A$.
\item $\sigma_{\Pi}(x,k)\equiv0\pmod{n}$ for all $x\in A$ if and only if $n$ divides $k$.
\item  $\sigma_{\Pi}(x,qm+r)\equiv q\sigma_{\Pi}(x,m)+\sigma_{\Pi}(x,r)\pmod{n}$ for all $x\in A$.
\item $\sigma_{\Pi}(x,k_1)\equiv\sigma_{\Pi}(x,k_2)\pmod{n}$ for all $x\in A$  if and only if $k_1\equiv k_2\pmod{n}$.
\end{enumerate}
\end{proposition}
\begin{proof}
(a) The proof is similar to that of Proposition~\ref{SIGMA}(a).

(b) By Proposition~\ref{SIGMA}(b), for any $x\in A$, we have $\sigma_{\Pi}(x,m)\equiv\sigma_{\pi}(x,m)\equiv0\pmod{m}$, so 
\[
\sigma_{\Pi}(x,n)\equiv\frac{n}{m}\sigma_{\Pi}(x,m)\equiv0\pmod{n}.
\]
If $n|k$, then for any $x\in A$, we have
\[
\sigma_{\Pi}(x,k)\equiv \frac{k}{n}\sigma_{\Pi}(x,n)\equiv0\pmod{n}.
\]

Conversely, if $\sigma_{\Pi}(x,k)\equiv0\pmod{n}$ for all $x\in A$, then by (a),
\[
k\equiv\sigma_{\Pi}(1,k)\equiv\sigma_{\Pi}(xx^{-1},k)\equiv\sigma_{\Pi}(x^{-1},\sigma_{\Pi}(x,k))\equiv0\pmod{n}.
\]

(c) Since $m=|\varphi|$, we have
\begin{align*}
\sigma_{\Pi}(x,qm+r)
&=\sum_{i=1}^{qm+r}\Pi(\varphi^{i-1}(x))=\sum_{i=1}^r\Pi(\varphi^{i-1}(x))+\sum_{i=r+1}^{qm+r}\Pi(\varphi^{i-1}(x))\\
&=\sum_{i=1}^r\Pi(\varphi^{i-1}(x))+\sum_{i=1}^{qm}\Pi(\varphi^{i-1}(x))
=\sum_{i=1}^r\Pi(\varphi^{i-1}(x))+q\sum_{i=1}^{m}\Pi(\varphi^{i-1}(x))\\
&=\sigma_{\Pi}(x,r)+q\sigma_{\Pi}(x,m)\pmod{n}.
\end{align*}

(d) This follows from (b) and (c).
\end{proof}

Furthermore, we can similarly establish the concepts of an average function and a mate function for the extended power function $\Pi:A\to\mathbb{Z}_n$ in the following manner. By Proposition~\ref{SIGMA}(b), we have $ \sigma_{\Pi}(x,m) \equiv 0 \pmod{m} $, for all $ x \in A $. Thus, we may define a function  $\Av:A\to\mathbb{Z}_{n/m}$   by 
 \begin{align}\label{Aver}
 \Av(x)=\frac{1}{m}\sigma_{\Pi}(x,m),\quad x\in A.
 \end{align}
Moreover, since $\Pi(x)\equiv\pi(x)\pmod{m}$, we may define a function $\Lambda:A\to\mathbb{Z}_{n/m}$ by
\[
\Lambda(x)=\frac{1}{m}(\Pi(x)-\pi(x)),\quad x\in A.
\]
The functions $\Av$ and $\Lambda$ shall be termed the \textit{average function} and the \textit{mate function} of $\Pi,$ respectively. The subsequent properties of $\Av$ and $\Lambda$ closely resemble those of $\av$ and $\lambda.$ The proof follows along the same lines as that of Proposition~\ref{av-lambda}.

\begin{proposition}\label{Av-mate}
Let $\varphi$ be a skew-morphism of a group $A$ with the associated power function $\pi:A\to\mathbb{Z}_m$,
where $m$ is the order of $\varphi$, and let $\Pi:A\to\mathbb{Z}_n$ be an extended power function
of $\varphi$, where $n$ is a positive multiple of $m$. Then the average function 
$\Av:A\to\mathbb{Z}_{n/m}$ and the mate function $\Lambda:A\to\mathbb{Z}_{n/m}$ 
of $\Pi$ have the following properties:
\begin{enumerate}[\rm(a)]
\item $\Av(x)\equiv\Av(\varphi(x))\pmod{n/m}$ for all $x\in A$.
\item $\Av$ is a homomorphism of $A$ into the multiplicative group $\mathbb{Z}_{n/m}^*.$
\item if $\varphi$ is an automorphism of $A$, then $\Lambda(xy)\equiv\Lambda(y)+\Lambda(x)\Av(y)\pmod{n/m}$
for all $x,y\in A$. In particular, if $\Av(x)\equiv1\pmod{n/m}$ for all $x\in A$, then
$\Lambda(xy)\equiv\Lambda(x)+\Lambda(y)\pmod{n/m}$ for all $x,y\in A$.
\end{enumerate}
\end{proposition}

\section{Extended skew products}\label{sec:esp}
In this section, we leverage the theory of extended power functions
 developed in the preceding section to establish the sought-after relationship 
between cyclic complementary extensions and skew-morphisms along 
with their associated extended power functions.

The first two statements of the following result already appeared in \cite[Lemma 4.1]{CJT2016}.
\begin{proposition}\label{one-direction}
Let $ G = A \langle c \rangle $ be a cyclic complementary  extension of a finite group $A$
 with a cyclic group $\langle c \rangle $ of order $n$, and let $ \varphi $ and $\Pi$ be 
 defined by \eqref{skew-def}, and let $m$ denote the order of $ \varphi $. Then following hold true:
 \begin{enumerate}[\rm(a)]
 \item $ \langle c^m \rangle $ is a normal subgroup of $G$,
and the factor group $ G/ \langle c^m \rangle $ is isomorphic to the skew product $ A \langle \varphi \rangle $.
\item The subgroup $ A \langle c^m \rangle =A \ltimes \langle c^m \rangle $ is a semidirect 
product in which 
\[
 x^{-1}c^mx=(c^m)^{\sigma_{\Pi}(x,m)/m},\quad\text{for all $ x \in A $}.
 \]
 \item  $\Av$ is trivial (namely, $\Av(x)\equiv1\pmod{n/m}$ for all $x\in A$) if and only if $c^m$
is a central element of $G$.
 \end{enumerate}
\end{proposition}

\begin{proof} (a) Let $ \Phi : A \langle c \rangle \to A \langle \varphi \rangle $ be the mapping defined by the 
formula $ \Phi(xc^i) = x\varphi^i $, for all $ x \in A $ and $ i \in {\Bbb Z}_n $. Then,
\[ 
\Phi(xc^iyc^j) = \Phi(x\varphi^i(y)c^{\sigma_{\Pi}(y,i)+j}) =x\varphi^i(y)\varphi^{\sigma_{\Pi}(y,i)+j}, 
\]
for all $ x,y \in A $ and $ i,j \in {\Bbb Z}_n $ (with the multiplication performed in 
$ A \langle c \rangle $), while
\[ 
\Phi(xc^i)\Phi(yc^j) = x\varphi^iy\varphi^j = x \varphi^i(y) \varphi^{\sigma_{\pi}(y,i)+j} ,
\]
for all $ x,y \in A $ and $ i,j \in {\Bbb Z}_n $ (with the multiplication performed in 
$ A \langle \varphi \rangle $). It follows that $ \Phi $ is a homomorphism from 
$ A \langle c \rangle $ to $ A \langle \varphi \rangle $ mapping $c$ to $\varphi$.
Since the order $m$ of $ \varphi $ divides the order $ n $ of $c$, $\Phi$ is surjective. 
Consequently, $ \Phi(c^m) = \varphi^m=
1_{A \langle \varphi \rangle} $, the kernel of $ \Phi $ is the group $ \langle c^m \rangle $ of order 
$ \frac{n}{m} $, and  $ A \langle \varphi \rangle \cong  A \langle c\rangle / \langle c^m \rangle $.

(b) Since $ \langle c^m \rangle $ is normal in $G$ and $ A \cap \langle c^m \rangle = 
\langle 1_G \rangle $, the product $ A \langle c^m \rangle $ is a semidirect product
$ A \ltimes \langle c^m \rangle $, and for all $x\in A$, we have
\[
c^mx = \varphi^m(x) c^{\sigma_{\Pi}(x,m)} = x (c^m)^{\sigma_{\Pi}(x,m)/m}.
\]

(c)  From (b) we see that $c^mx=xc^{m\Av(x)}$ for all $x\in A$. Thus, $c^m$ is central in $G$
if and only if $\Av(x)\equiv1\pmod{n/m}$ for all $x\in A$; that is, $\Av$ is trivial. 
\end{proof}

We are now ready to prove Theorem~\ref{main-thm1}, the first main result of this paper.
\begin{proof}[Proof of Theorem~\ref{main-thm1}]
Suppose that $G=A\langle c\rangle$ represents a cyclic complementary extension of $A$ 
by a cyclic group $\langle c\rangle$ of order $n$. In a manner akin to the outset of 
Section 3, it is possible to deduce, from the identity~\eqref{Comm}, a pair $(\varphi,\Pi)$ 
comprising a skew-morphism $\varphi$ of $A$ and an extended power function $\Pi:A\to\mathbb{Z}_n$.

Conversely, suppose that $\varphi$ is a skew-morphism of $A$, and $\Pi:A\to\mathbb{Z}_n$
is an extended power function of $\varphi$, where $n$ is a positive multiple of  the order  $m$ of $\varphi$. 
Define a binary operation $ \ast$ on the set $G:=A\times C$ via the formula:
\begin{equation}\label{def}
(x,c^i) \ast (y,c^j) = (x\varphi^i(y), c^{\sigma_{\Pi}(y,i)+j}),\quad\text{$x,y\in A $ and $ i,j \in {\Bbb Z}_n $.}
\end{equation}
It is clear that $\ast$ is a well-defined binary operation on the set 
$ A \times \langle c \rangle $. To show that $A\times\langle c\rangle$ is a group, it
suffices to verify the associativity of the operation,
the existence of a left identity element and the existence of 
a left inverse for any element in $A\times\langle c\rangle$.

Let us begin with the associativity. Let $ x,y,z \in A $ and $ i,j,k \in {\Bbb Z}_n $. Then,
\begin{align*}
((x,c^i) \ast (y,c^j)) \ast (z,c^k) &= (x \varphi^i(y), c^{\sigma_{\Pi}(y,i)+j}) \ast (z,c^k) \\
&= (x \varphi^i(y) \varphi^{\sigma_{\Pi}(y,i)+j}(z), c^{\sigma_{\Pi}(z,\sigma_{\Pi}(y,i)+j)+k}),
\end{align*}
while
\begin{align*}
(x,c^i) \ast((y,c^j) \ast (z,c^k)) 
&= (x,c^i) \ast (y \varphi^j(z),c^{\sigma_{\Pi}(z,j)+k}) \\
&= (x \varphi^i(y \varphi^j(z)), c^{\sigma_{\Pi}(y \varphi^j(z),i)+\sigma_{\Pi}(z,j)+k}) \\
&= (x \varphi^i(y) \varphi^{\sigma_{\Pi}(y,i)+j}(z), c^{\sigma_{\Pi}(y \varphi^j(z),i)+\sigma_{\Pi}(z,j)+k}) .
\end{align*}
To complete the proof of associativity we need to prove the congruence 
\[
 \sigma_{\Pi}(z,\sigma_{\Pi}(y,i)+j)+k\equiv \sigma_{\Pi}(y \varphi^j(z),i)+\sigma_{\Pi}(z,j)+k  \pmod{n} .
\] 
Using Proposition~\ref{SIGMA2}(a), we have
\begin{align*}
\sigma_{\Pi}(y \varphi^j(z),i)+\sigma_{\Pi}(z,j)+k 
&= \sigma_{\Pi}(\varphi^j(z),\sigma_{\Pi}(y,i))+\sigma_{\Pi}(z,j)+k \\
&= \sum_{s=1}^{\sigma_{\Pi}(y,i)} \Pi(\varphi^{j+s-1}(z))+\sum_{t=1}^j \Pi(\varphi^{t-1}(z)) + k \\
&=\sum_{t=j+1}^{\sigma_{\Pi}(y,i)+j} \Pi(\varphi^{t-1}(z))+\sum_{t=1}^j \Pi(\varphi^{t-1}(z)) + k\\
&=\sum_{t=1}^{\sigma_{\Pi}(y,i)+j} \Pi(\varphi^{t-1}(z)) + k \\
&= \sigma_{\Pi}(z,\sigma_{\Pi}(y,i)+j)+k \pmod{n}.
\end{align*}

Moreover, for an arbitrary element  $ (x,c^i) $ of $A \times \langle c \rangle$, we have
\[
 (1_A,c^0) \ast (x,c^i) = (1_A \varphi^0(x), c^{\sigma_{\Pi}(x,0)+i}) = (x,c^i),
 \]
and
 \begin{align*}
(\varphi^{-i}(x^{-1}),c^{\sigma_{\Pi}(x^{-1},n-i)}) \ast (x,c^i)
&= (\varphi^{-i}(x^{-1})\varphi^{\sigma_{\Pi}(x^{-1},n-i)}(x), c^{\sigma_{\Pi}(x,\sigma_{\Pi}(x^{-1},n-i))+i} ) \\
&=(\varphi^{n-i}(x^{-1})\varphi^{\sigma_{\Pi}(x^{-1},n-i)}(x), c^{\sigma_{\Pi}(xx^{-1},n-i)+i} )\\
&=(\varphi^{n-i}(x^{-1}x), c^{n-i+i} ) \\
&=(\varphi^{n-i}(1_A),c^0)=(1_A, c^0 ),
\end{align*}
so the binary operation $ \ast $ has a left identity $ (1_A,c^0) $, and the element
 $(x,c^i)$ has a left inverse $(\varphi^{-i}(x^{-1}),c^{\sigma_{\Pi}(x^{-1},n-i)})$. 
 Therefore, $A\times\langle c\rangle$ is group with respect to the operation.
 
To prove the second part of the theorem, suppose first that there exists an automorphism $\theta$ of $A$ such that $\varphi_2=\theta\varphi\theta^{-1}$ 
and $\Pi_2=\Pi_1\theta$. Define a mapping 
\[
\Theta: A\langle c_1\rangle\to A\langle c_2\rangle, xc_1^i\mapsto \theta(x)c_2^i.
\]
It is evident that  $\Theta$ is surjective with $\Theta(A)=A$ and $\Theta(c_1)=c_2$.
Now for any $x,y\in A$ and $i,j\in\mathbb{Z}_n$, we have
\begin{align*}
\Theta(xc_1^iyc_1^j)&=\Theta(x\varphi_1^i(y)c_1^{\sigma_{\Pi_1}(y,i)+j})\\
&=\theta(x\varphi_1^i(y))c_2^{\sigma_{\Pi_1}(y,i)+j}\\
&=\theta(x)(\theta\varphi_1\theta^{-1})^i(\theta(y))c_2^{\sigma_{\Pi_1}(y,i)+j}\\
&=\theta(x)\varphi_2^{i}(\theta(y))c_2^{\sigma_{\Pi_1}(y,i)+j}\\
&=\theta(x)\varphi_2^{i}(\theta(y))c_2^{\sigma_{\Pi_2}(\theta(y),i)+j}\\
&=\theta(x)c_2^i\theta(y)c_2^j=\Theta(xc_1^i)\Theta(yc_1^j).
\end{align*}
 Therefore, $\Theta$ is the required isomorphism. 

Conversely, suppose that  there exists an isomorphism $\Theta:A\langle c_1\rangle\to A\langle c_2\rangle$ 
such that $\Theta(A)=A$ and  $\Theta(c_1)=c_2$. Then the restriction $\theta=\Theta|_A$ is an automorphism
of $A$. For all $x\in A$, we have
\begin{align*}
\Theta(\varphi_1(x)c_1^{\Pi_1(x)})
&=\Theta(c_1x)=\Theta(c_1)\Theta(x)=c_2\theta(x)=\varphi_2\theta(x)c_2^{\Pi_2(\theta(x))}
=\Theta\big(\theta^{-1}\varphi_2\theta(x)c_2^{\Pi_2(\theta(x))}\big),
\end{align*}
so $\varphi_1(x)=\theta^{-1}\varphi_2\theta(x)$ and $\Pi_1(x)\equiv\Pi_2(\theta(x))\pmod{n}$, as required.
\end{proof}

The cyclic complementary extension $G$ of $A$, which we construct in the proof of 
Theorem~\ref{main-thm1} using a skew-morphism $\varphi$ and an extended power function 
$\Pi$ of $\varphi$, will be termed the \textit{extended skew product} of $A$ determined by 
$\varphi$ and $\Pi$, and it will be denoted by $\Ext(A,\varphi,\Pi)$.
For simplicity, we shall represent the element $(x,c^i)$ of $\Ext(A,\varphi,\Pi)$ as $xc^i$, and express 
the multiplication as follows:
 \[
 xc^i yc^j=x\varphi^i(y)c^{\sigma_{\Pi}(y,i)+j},\quad \text{for all $x,y\in A$ and $i,j\in\mathbb{Z}_n$}.
 \]

Consider a skew-morphism $\varphi$ of $A$, and let $\Pi:A\to\mathbb{Z}_n$ be an extended power 
function of $\varphi$. Due to the second part of Theorem~\ref{main-thm1}, it follows that for any automorphism 
$\theta$ of $A$, the permutation $\varphi':=\theta\varphi\theta^{-1}$ is also a skew-morphism of $A$, 
and $\Pi':=\Pi\theta^{-1}$ is  an extended power function of $\varphi'$. Consequently, the automorphism 
group $\Aut(A)$ of $A$ acts on the set $\Skew(A)$ of all skew-morphisms of $A$ 
via conjugation. Two skew-morphisms of $A$ are considered \textit{equivalent} if they fall within the same orbit of the above
 action. Inspired by Theorem~\ref{main-thm1}, we refer to two cyclic complementary extensions, $A\langle c_1\rangle$ and $A\langle c_2\rangle$, as \textit{equivalent} if there exists an isomorphism $\Theta$ from $A\langle c_1\rangle$ to $A\langle c_2\rangle$ such that $\Theta(A)=A$ and $\Theta(c_1)=c_2.$
It is worth noting that this particular case of bijective correspondence between equivalence classes of skew-product groups and equivalence classes of skew-morphisms was previously established in~\cite[Proposition~3.2.9]{Bachraty2020}.

\section{Application}\label{sec:app}
In this section, as an application of the theory developed in the previous sections, we 
determine the cyclic complementary extensions of cyclic groups corresponding to automorphisms. 

According to Theorem~\ref{main-thm1}, when dealing with an automorphism $\varphi$ of a cyclic group
 $A$ of order $k$ (expressed additively as $\mathbb{Z}_n$ in the following discussion), the key 
 to determining the corresponding cyclic complementary extensions of $A$ is to establish the 
 extended power functions of $\varphi$. To achieve this, we rely on a technical result involving 
 a function $\tau:R\times\mathbb{N}\to R$, where $R$ represents a ring with an identity. 
 This function is defined as follows:
\[
\tau(t,0)=0,\quad\text{and}\quad \tau(t,k)=\sum_{i=1}^kt^{i-1},\quad k>0.
\]
\begin{lemma}\cite[Lemma 2]{HR2022}\label{TECH}
The function $\tau(t,x)$ satisfies the following equalities:
\begin{align}\label{GEOM}
\tau(t,x+y)=\tau(t,x)+t^x\tau(t,y)\quad\text{and}\quad\tau(t,xy)=\tau(t,x)\tau(t^x,y),\quad x,y\in\mathbb{N}.
\end{align}
\end{lemma}

The extended power  functions of automorphisms of $\mathbb{Z}_k$
are determined in the following.
 
\begin{lemma}\label{ext-auto}
Let $\varphi: x\mapsto xr$ be an automorphism of $\mathbb{Z}_k$ 
determined by the number  $r\in\mathbb{Z}_k^*$, and let $n$ be a
 positive multiple  of  the multiplicative order $m$ of $r$ in $\mathbb{Z}_k$. 
Then the  extended power functions $\Pi:\mathbb{Z}_k\to\mathbb{Z}_n$  of $\varphi$  are  determined by
the formula
\begin{equation}\label{nth-EPF}
\Pi_{r,s,t}(x)\equiv1+ms\sum_{i=1}^xt^{i-1}\pmod{n},\quad x\in\mathbb{Z}_k,
\end{equation}
where the numbers $s\in\mathbb{Z}_{n/m}$ and $t\in\mathbb{Z}_{n/m}^*$ satisfy the conditions:
\begin{enumerate}[\rm(a)]
\item $t^{r-1}\equiv 1\pmod{n/m}$,
\item $s\sum\limits_{i=1}^kt^{i-1}\equiv0\pmod{n/m}$,
\item $s\sum\limits_{i=1}^m(\sum\limits_{j=1}^rt^{j-1})^{i-1}\equiv t-1\pmod{n/m}$.
\end{enumerate}
Moreover, for fixed $r\in\mathbb{Z}_k^*$, $\Pi_{r,s,t}=\Pi_{r,s',t'}$ if and only if
\begin{equation}\label{st-eqn}
s\equiv s'\pmod{n/m}\quad\text{and}\quad s(t-t')\equiv0\pmod{n/m}.
\end{equation}
\end{lemma}
\begin{proof}
Suppose that $\Pi:\mathbb{Z}_k\to\mathbb{Z}_n$ is an extended power function of an automorphism 
$\varphi$ of $\mathbb{Z}_k$ taking $x\mapsto rx$, where $r\in\mathbb{Z}_k^*$ and $n$
is positive multiple of the multiplicative order $m$ of $r$ in $\mathbb{Z}_k$. By 
Proposition~\ref{Av-mate}(b)-(c), the function $\Av:\mathbb{Z}_k\to\mathbb{Z}_{n/m}$
 is a homomorphism from the additive group $\mathbb{Z}_k$ into the  multiplicative 
 group $\mathbb{Z}_{n/m}^*,$  and   the function $\Lambda:\mathbb{Z}_k\to\mathbb{Z}_{n/m}$ 
satisfies the equation 
\[
\Lambda(x+y)=\Lambda(x)+\Lambda(y)\Av(x)\pmod{n/m},\quad\text{$x,y\in\mathbb{Z}_k$}.
\] 
Set $s=\Lambda(1)\in\mathbb{Z}_{n/m}$ and  $t=\Av(1)\in\mathbb{Z}_{n/m}^*$. Then 
\[
\Av(x)\equiv \Av(1)^x=t^x\pmod{n/m},\quad \text{for all $x\in\mathbb{Z}_k$}.
\]
 By induction we obtain
\begin{align*}
\Lambda(x)&=\Lambda((x-1)+1)\equiv\Lambda(x-1)+\Lambda(1)\Av(x-1)\\
&\equiv s\sum_{i=1}^{x-1}t^{i-1}+st^{x-1}=s\sum_{i=1}^xt^{i-1}=s\tau(t,x)\pmod{n/m}.
\end{align*}
Therefore, $\Pi(x)\equiv \pi(x)+m\Lambda(x)\equiv1+m\Lambda(x)=1+ms\tau(t,x)\pmod{n}.$

The numerical conditions can be easily derived  as follows:
\begin{align*}
t&\equiv\Av(1)\equiv\Av(\varphi(1))\equiv\Av(r)\equiv t^r\pmod{n/m},\\
1&\equiv\Pi(0)\equiv\Pi(k)\equiv1+ms\sum_{i=1}^kt^{i-1}\pmod{n},\\
t&\equiv\Av(1)\equiv\frac{1}{m}\sum_{i=1}^m\Pi(\varphi^{i-1}(1))
\equiv\frac{1}{m}\sum_{i=1}^m\Pi(r^{i-1})\\
&\equiv\frac{1}{m}\sum_{i=1}^m\big(1+ms\tau(t,r^{i-1})\big)\stackrel{\eqref{GEOM}}
\equiv\frac{1}{m}\sum_{i=1}^m\big(1+ms\tau(t,r)^{i-1}\big)\\
&\equiv 1+s\sum_{i=1}^m\tau(t,r)^{i-1}\pmod{n/m}.
\end{align*}

It remains to verify that $\Pi_{r,s,t}$ is indeed an extended power function of $\varphi$,
provided the required numerical conditions are satisfied. It suffices to
verify that it fulfills the conditions (a)--(c) in Definition~\ref{n-power}. The first two
conditions are clearly satisfied. In what follows, we verify the last one. Write $\Pi:=\Pi_{r,s,t}$. 

First observe that, since $|\varphi|=m$, we have 
\[
\sigma_{\Pi}(x,qm+r)\equiv q\sigma_{\Pi}(x,m)+\sigma_{\Pi}(x,r)\pmod{n}, \quad x\in\mathbb{Z}_k.
\]
 Then, using Lemma~\ref{TECH} and the numerical conditions (a) and (c), we have 
\begin{align*}
\frac{1}{m}\sigma_{\Pi}(x,m)
&\equiv\frac{1}{m}\sum_{i=1}^m\Pi(\varphi^{i-1}(x))\equiv\frac{1}{m}\sum_{i=1}^m\Pi(xr^{i-1})\\
&\equiv \frac{1}{m}\sum_{i=1}^m\big(1+ms\tau(t,xr^{i-1})\big)\stackrel{\eqref{GEOM}}\equiv\frac{1}{m}\sum_{i=1}^m\big(1+ms\tau(t,x)\tau(t,r)^{i-1}\big)\\
&\equiv 1+s\tau(t,x)\sum_{i=1}^m\tau(t,r)^{i-1}\equiv 1+\tau(t,x)(t-1)\equiv t^x\pmod{n/m}.
\end{align*}
Thus, for any $x,y\in\mathbb{Z}_k$, we have
\begin{align*}
\Pi(x+y)&\equiv1+ms\tau(t,x+y)\equiv1+ms\big(\tau(t,y)+\tau(t,x)t^y\big)\\
&\equiv\sigma_{\Pi}(y,1)+s\tau(t,x)\sigma_{\Pi}(y,m)\equiv \sigma_{\Pi}(y,1+ms\tau(t,x))\\
&\equiv\sigma_{\Pi}(y,\Pi(x))\pmod{n}.
\end{align*}
Hence, the last condition is also fulfilled. Consequently,
$\Pi_{r,s,t}$ is an extended power function of $\varphi$.

Finally, for fixed $r\in\mathbb{Z}_k^*$, if $\Pi_{r,s,t}(x)=\Pi_{r,s',t'}(x)$ for any $x\in\mathbb{Z}_k$, then 
\[
1+ms\sum_{i=1}^xt^{i-1}\equiv\Pi_{r,s,t}(x)\equiv\Pi_{r,s',t'}(x)\equiv1+ms'\sum_{i=1}^x{t'}^{i-1}\pmod{n},
\]
and so 
\[
s\sum_{i=1}^xt^{i-1}\equiv s'\sum_{i=1}^x{t'}^{i-1}\pmod{n/m}.
\]
 Putting $x=1$ and $x=2$, we get~\eqref{st-eqn}. 
Conversely, if the conditions in~\eqref{st-eqn} are satisfied, it is evident that $\Pi_{r,s,t}=\Pi_{r,s',t'}$,
 as required.
\end{proof}
We are finally ready to prove Theorem~\ref{main-thm2}, the second main result of our paper.
\begin{proof}[Proof of Theorem~\ref{main-thm2}]
Note that every automorphism of the cyclic group  $A:=\langle a|a^k=1\rangle$
 is determined by an assignment of the form $a^x\mapsto a^{rx}$, where $r\in\mathbb{Z}_k^*$.
 Let $\varphi:x\mapsto rx$ be the corresponding automorphism of the additive group $\mathbb{Z}_k$
 which is isomorphic to $A$, let $m$ denote the order of $\varphi$, which is equal to 
 the multiplicative order of $r$ in $\mathbb{Z}_k$, and let $n$ be 
 a positive multiple of $m$. By Theorem~\ref{main-thm1}, every cyclic complementary extension $G$
corresponding to $\varphi$ can be constructed from the identity
\[
ca^x=a^{\varphi(x)}c^{\Pi(x)},\quad x\in\mathbb{Z}_k,
\]
where $c$ is a generator of a cyclic group $C:=\langle c|c^n=1\rangle$ of order $n$, and $\Pi:\mathbb{Z}_k\to\mathbb{Z}_{n/m}$ is an extended power function of $\varphi$ determined in Theorem~\ref{main-thm2}. Thus, $G=\langle a,c\rangle$ and we have the following relations in $G$: 
\[
a^k=c^n=1,\quad c^ma=a^{m\Av(1)}=a^{mt},\quad ca=a^{\varphi(1)}c^{\Pi(1)}=a^rc^{1+ms}.
\]
It is apparent that these relations are sufficient to define the required group $G$.
\end{proof}

\begin{example}The assignment $\varphi:a\mapsto a^3$ extends to
an automorphism of the cyclic group $A=\langle a|a^8=1\rangle$. 
The cyclic complementary extensions of $A$ by another cyclic group $C=\langle c|c^8=1\rangle$ 
of order $8$ corresponding to the automorphism $\varphi$ are summarized in Table~\ref{TAB}, where
$(r,s,t)$ are all the triples satisfying the three conditions (a)--(c) stated in
Theorem~\ref{ext-auto}. 
\medskip

\begin{center}\label{1}
\begin{threeparttable}[b]
\caption{Cyclic complementary extensions of $A=\langle a|a^8=1\rangle$}\label{TAB}
\begin{tabular*}{140mm}[c]{p{13mm}p{27mm}p{14mm}p{71mm}}
\toprule
$(r,s,t)$    & $\Pi(x)\pmod{8}$                            &$\varphi$    &$A\langle c\rangle$                        \\
\hline
  $(3,0,1)$ &$1$ &  $a\mapsto a^3$  & $\langle a,c\mid a^8=c^8=1, a^c=a^3\rangle$\\

   $(3,1,1)$ &$1+2x$ &  $a\mapsto a^3$  & $\langle a,c\mid a^8=c^8=1, ca=a^3c^5, ca^3=ac^7\rangle$\\

   $(3,2,1)$ &$1+4x$ &  $a\mapsto a^3$  & $\langle a,c\mid a^8=c^8=1, ca=a^3c^5, ca^3=ac^5\rangle$\\

   $(3,3,1)$ &$1+6x$ &  $a\mapsto a^3$  & $\langle a,c\mid a^8=c^8=1,ca=a^3c^7, ca^3=ac^3\rangle$\\

   $(3,1,3)$ &$1+2{\scriptstyle\sum_{i=1}^x}3^{i-1}$ &  $a\mapsto a^3$  & $\langle a,c\mid a^8=c^8=1, ca=a^3c^3,ca^3=ac^3\rangle$\\

   $(3,3,3)$ &$1+6{\scriptstyle\sum_{i=1}^x}3^{i-1}$ &  $a\mapsto a^3$  & $\langle a,c\mid a^8=c^8=1, ca=a^3c^7, ca^3=ac^7\rangle$\\
\bottomrule
\end{tabular*}
\end{threeparttable}
\end{center}
\end{example}
\medskip

\section{Final comments and open problems}
Let $\varphi$ be a skew-morphism of a finite group $A$ and let $\pi: A \to \mathbb{Z}_m$ be its associated power function, where $m = |\varphi|$. It is well-known that if $A$ is a nontrivial group, then the subgroup $\Ker\pi$ is nontrivial~\cite[Theorem 4.3]{CJT2016}. Additionally, if $A$ is abelian, then $\varphi(\Ker\pi) = \Ker\pi$, that is, $\varphi$ preserves the kernel of $\pi$~\cite[Lemma 5.1]{CJT2007}.

Now, consider an extended power function $\Pi: A \to \mathbb{Z}_n$ of $\varphi$. As we have already pointed out, 
the kernel $\Ker\Pi$ and the core $\Core\Pi$ of $\Pi$ form subgroups of $A$; in particular $\Ker\Pi\leq\Ker\pi$ and $\Core\Pi\leq\Core\pi$. It has been shown that if $G=A\langle c\rangle$
is the corresponding cyclic complementary extension of $A$, then $\Ker\Pi=A\cap A^c$ and $\Core\Pi=\cap_{i=1}^nA^{c^i}$~\cite{HZ2023}, which are natural generalization of results on $\Ker\pi$ and $\Core\pi$ obtained earlier~\cite{CJT2016,HNWY2019}. However, it should be noted that the same arguments as those
used previously cannot be extended to establish that $\Ker\Pi$ is nontrivial when $A$ is nontrivial, or that $\varphi(\Ker\Pi) = \mathrm{Ker}(\Pi)$ when $A$ is abelian.

In the study of extended power functions $\Pi$ associated with a skew-morphism $\varphi$ of a finite group 
$A$, the primary objective should be to analyze the characteristics and properties of the subgroups $\Ker\Pi$ and $\Core\Pi$. This problem can be formalized as follows:
\begin{problem}
Investigate the properties of the subgroups $\Ker\Pi$ and $\Core\Pi$ within the framework of an extended power function $\Pi$ associated with a skew-morphism $\varphi$ of a finite group $A$.
\end{problem}

Upon reviewing the proof of Theorem~\ref{main-thm2}, it becomes evident that the average function $\Av$ and the mate function $\Lambda$ played pivotal roles in the process of determining extended power functions. This observation prompts an intriguing exploration of a more expansive theory and potential applications for these functions, along with their counterparts $\av$ and $\lambda,$ which are associated with the power function $\pi$.

\begin{problem}
Investigate the average functions $\av:A\to\mathbb{Z}_{m/p}$ and $\Av:A\to\mathbb{Z}_{n/m}$, 
and the mate functions $\lambda:A\to\mathbb{Z}_{m/p}$ and $\Lambda:A\to\mathbb{Z}_{n/m}$
associated with $\pi$ and $\Pi$, respectively.
\end{problem}

Recall that a skew-morphism $\varphi$ is referred to as smooth if $\pi(\varphi(x)) \equiv \pi(x) \pmod{m}$, 
for all $x \in A$. Notably, smooth skew-morphisms serve as significant generalizations of group automorphisms which are at the same time still relatively close to them~\cite{BJ2017, WHYZ2019}. In view of this, the 
concept of a skew-morphism being smooth with respect to its power function can be generalized to extended
power functions as follows: An extended power function $\Pi: A \to \mathbb{Z}_n$ associated with a skew-morphism $\varphi$ is termed \textit{smooth} if
\[
\Pi(\varphi(x))\equiv\Pi(x)\pmod{n},\quad\text{for all $x\in A$}.
\]
It is evident that if an extended power function $\Pi$ of a skew-morphism $\varphi$ is smooth, then
the skew-morphism itself is also smooth. However, the converse is not true. In view of this, it seems interesting
to investigate the following problem: 
\begin{problem}
Investigate extended power functions of smooth skew-morphisms.
\end{problem}

The classification problem concerning skew-morphisms of dihedral groups and their associated skew-product groups has been a longstanding challenge. This has, however, changed recently with the progress made in~\cite{HKK2023}. Expanding the theory developed in~\cite{HKK2023}, it is intriguing to further extend the obtained results and address the following classification problem aiming to provide a deeper understanding of groups that can be factored into a product of two disjoint subgroups, with one subgroup being dihedral and the other subgroup 
being cyclic.
\begin{problem}
Classify cyclic complementary extensions of dihedral groups.
\end{problem}
\bigskip

\begin{center}
{\bf Acknowledgement}
\end{center}
The first author is supported by National Natural Science Foundation of China (11801507)
and the Slovenian Research Agency (N1-0208). 
The second author is supported by VEGA Research Grant 1/0437/23 and APVV Research Grant 19-0308.


\end{document}